\documentclass[11pt]{article}
\usepackage[bookmarks, colorlinks=true, plainpages = false, citecolor = blue,linkcolor=red,urlcolor = blue, filecolor = blue,pagebackref]{hyperref}

\usepackage{lipsum}
\usepackage{amsfonts}
\usepackage{graphicx}
\usepackage{epstopdf}

\usepackage{mathtools}
\usepackage{empheq}
\usepackage{amsfonts}
\usepackage{amsgen,amsmath,amsthm,amstext,amsbsy,amsopn,amssymb,subfigure,stmaryrd}
\usepackage{comment}
\usepackage{enumitem}
\usepackage{booktabs}
\usepackage{blkarray}
\usepackage{algorithm}
\usepackage{algpseudocode}
\usepackage{url}
\usepackage{longtable}
\usepackage{mathtools}
\usepackage{multirow}

\ifpdf
\DeclareGraphicsExtensions{.eps,.pdf,.png,.jpg}
\else
\DeclareGraphicsExtensions{.eps}
\fi

\newtheorem{Theorem}{Theorem}

\newtheorem{Lemma}{Lemma}

\DeclareMathAlphabet\mathbfcal{OMS}{cmsy}{b}{n}

\newcommand{\be}{\begin{equation}}
\newcommand{\ee}{\end{equation}}
\newcommand{\bea}{\begin{eqnarray}}
\newcommand{\eea}{\end{eqnarray}}
\newcommand{\beas}{\begin{eqnarray*}}
	\newcommand{\eeas}{\end{eqnarray*}}

\newcommand{\bbR}{\mathbb{R}}

\newcommand{\tF}{{\rm F}}

\newcommand{\A}{{\mathbf{A}}}

\newcommand{\cN}{{\cal N}}

\newcommand{\dist}{{\rm dist}}

\newcommand{\argmin}{\mathop{\rm arg\min}}
\newcommand{\argmax}{\mathop{\rm arg\max}}

\newcommand{\bbP}{\mathbb{P}}
\newcommand{\bbE}{\mathbb{E}}
\newcommand{\bbO}{\mathbb{O}}

\makeatletter
\newcommand*{\rom}[1]{\expandafter\@slowromancap\romannumeral #1@}
\makeatother

\allowdisplaybreaks

\begin{document}

\title{Discussion of ``Vintage factor analysis with Varimax performs statistical inference"\footnote{Supported in part by NSF Grant CAREER-2203741 and NIH Grant R01 GM131399.}}

\author{Rungang Han\footnote{Department of Statistical Science, Duke University} ~ and ~ Anru R. Zhang\footnote{Departments of Biostatistics \& Bioinformatics, Computer Science, Mathematics, and Statistical Science, Duke University}}

\date{(\today)}

\maketitle
We wholeheartedly congratulate Drs. Rohe and Zeng for their insightful paper  \cite{rohe2020vintage} on vintage factor analysis with Varimax rotation. Varimax rotation is a basic scheme to simplify the expression of a particular subspace and is included in build-in standard packages \emph{stats} in R and \emph{PROC FACTOR} statement in SAS. Drs. Rohe and Zeng nicely show that the principal component analysis with Varimax rotation actually performs statistical inference for the explainable factors. 

Drs. Rohe and Zeng suggested leptokurtosis as a key identifiability condition for Varimax rotation; the number of factors often increases as the data dimension and sample size grow. It is thus natural to ask whether Varimax works with vanishing leptokurtosis and/or a growing number of factors. This note discusses when Varimax recovers the subspace rotation in such high-dimensional regimes. As a first step, we assume the factor matrix $Z \in \bbR^{n\times k}$ includes a collection of i.i.d. centered random variables satisfying
\begin{equation}\label{eq:X-moment}
    \mathbb{E}Z_{ij}= \mathbb{E}Z_{ij}^3 =0, \quad \mathbb{E}Z_{ij}^2 = 1, \quad \mathbb{E}Z_{ij}^4 = \kappa > 3.
\end{equation}
We also assume $Z_{ij}$'s are sub-Gaussian such that $\bbE \exp\left(\lambda Z_{ij}\right) \leq e^{c\lambda^2}, \forall \lambda \in \bbR$ for some constant $c>0$. Let $\hat Z\in \mathbb{R}^{p\times r}$ be the observed factors generated as
\begin{equation*}
    \hat Z = Z R^*,
\end{equation*}
where $R^*$ is an unknown $k$-dimensional orthogonal matrix that represents the rotation to be recovered. Due to vanishing mean of $Z$, we focus on the following centered Varimax:
\begin{equation*}
    \hat R = \argmax_{R \in \mathcal O(k)} \sum_{i=1}^n\sum_{j=1}^k \left((\hat Z R^\top)_{ij}\right)^4,
\end{equation*}
where $\mathcal O(k)$ is the set of $k$-by-$k$ orthogonal matrices. Consider the following error metric:
\begin{equation*}
\dist(R^*,\hat R) = \min_{P \in \mathcal{P}(k)}\frac{\|\hat{R} - PR^*\|_\tF}{\|R^*\|_\tF} = k^{-1/2}\min_{P \in \mathcal{P}(k)} \|\hat{R} R^{*\top} - P\|_\tF,
\end{equation*}
where $\mathcal P(k) = \mathcal{O}(k)\cap \{0, \pm1\}^{k^2}$ is the set of orthogonal matrices that allow for column reordering and sign changes as defined in Eqn. (12) in \cite{rohe2020vintage}. The following theorem characterizes the conditions of $n, k, \kappa$ under which Varimax works or fails.
\begin{Theorem}\label{thm:nonasymptotics}
    Let $\delta \in (0,1/2]$ be any fixed value and $\kappa \leq C_0$ for some universal constant $C_0 > 3$.
    \begin{itemize}
        \item If $n \gtrsim \max\left\{\frac{k\log n}{(\kappa-3)^2}, \frac{k^2\log^2 n }{\kappa - 3}\right\}$, then
        \begin{equation*}
            \limsup_{n \rightarrow \infty} \bbP\left(\dist(R^*, \hat R) \geq \delta\right) = 0; 
        \end{equation*}
        \item if $k < n \lesssim k^2$, then
        \begin{equation*}
            \liminf_{n, k \rightarrow \infty} \bbP\left(\dist(R^*, \hat R) \geq \delta\right) = 1.
        \end{equation*}
    \end{itemize}
\end{Theorem}
If the kurtosis $\kappa>3$ is a constant, Theorem \ref{thm:nonasymptotics} suggests that a necessary and nearly sufficient condition for consistent rotation recovery is $n \gtrsim k^2$ -- this bound is tight up to $\log^2(n)$. Theorem \ref{thm:nonasymptotics} also shows when the leptokurtosis of factors is insignificant, i.e., $\kappa \rightarrow 3^+$, a sufficient sample size to ensure rotation recovery by Varimax is $\max\{k\log(n)/(\kappa-3)^2, k^2\log^2(n)/(\kappa-3)\}$, while it is unclear if this bound is sharp. It is of future interest to investigate the tight condition that guarantees the consistency of Varimax.

\begin{proof}[Proof of Theorem \ref{thm:nonasymptotics}]
For any random matrix $Z$ and fixed $a \in \bbR^k$, define the following stochastic functions
\begin{equation}\label{eq:mini-varimax}
    v(Z;a) = \frac{1}{n}\sum_{i=1}^n\left(\sum_{j=1}^ka_{j}Z_{ij}\right)^4.
\end{equation}
Note that the Varimax object function can be then written as $v(\hat Z; R^\top) = \sum_{j=1}^kv(\hat Z; R_{j:})$ where $R_{j:}$ if the $j$th row of $R$. Our proof relies on the following lemmas.
\begin{Lemma}\label{lm:metric-equivalence}
    Suppose $A \in \mathcal O(k)$ and $\min_{P \in \mathcal P(k)}\left\|A - P\right\|_\tF > t$ for some $t>0$. Then,
    \begin{equation*}
        \sum_{i=1}^k\left(\sum_{j=1}^k A_{ij}^4 - 1\right) \leq -\frac{1}{16}t^2.
    \end{equation*}
\end{Lemma}
\begin{Lemma}\label{lm:concentration-varimax}
    Suppose $Z$ is a $n$-by-$k$ random matrix with i.i.d. mean-zero, unit-variance and sub-Gaussian entries. Then, with probability at least $1-n^{-ck}-n^{-4}$, 
    \begin{equation*}
        \sup_{a: \|a\|_2 \leq 1} v(Z;a) - \bbE v(Z;a) \lesssim \sqrt{\frac{k\log n}{n}} + \frac{k^2\log^2 n}{n}.
    \end{equation*}
\end{Lemma}
\begin{Lemma}\label{lm:concentration-fourth-moment}
	Suppose $X_1,\ldots, X_p$ are independent random variables with mean zero, variance one, and sub-Gaussian tails. Then there exists constants $C, c>0$ such that 
	\begin{equation*}
	\bbP\left(\left|\sum_{i=1}^p (X_i^4 - \mathbb{E} X_i^4)\right| \geq C\sqrt{pt} + Ct^2 \right) \leq 2\exp(-ct). 
	\end{equation*}
\end{Lemma}
We start by proving the sufficient condition. Define
\begin{equation}\label{eq:def-delta}
    \delta := \frac{C_0}{\sqrt{\kappa-3}}\max\left\{\left(\frac{k^2\log^2 n}{n}\right)^{1/2}, \left(\frac{k\log n}{n}\right)^{1/4}\right\},
\end{equation}
for some positive constant $C_0$ to be specified later. It suffices to show that with high probability, for any $R \in \mathcal O(k)$ such that $\min_{P \in \mathcal P(k)}\left\|R^*R^\top - P\right\|_\tF > \delta k^{1/2}$, $R$ cannot be the solution of Varimax.

By Lemma \ref{lm:concentration-varimax}, we know with probability at least $1-o(1/n)$,
\begin{equation}\label{ineq:concentration-varimax}
    \sup_{a: \|a\|_2 \leq 1} v(Z;a) - \bbE v(Z;a) \leq C\left(\sqrt{\frac{k\log n}{n}} + \frac{k^2\log^2 n}{n}\right) \leq c(\kappa-3),
\end{equation}
where $c$ is some small positive constant. In the meantime,  Define $A =  RR^{*\top} = (a_1,a_2,\ldots,a_k)$, where $a_j$ is the $j$-th orthogonal column of $A$. 
Note that the Varimax loss function can be written as $v(\hat Z; R) = \sum_{j=1}^k v(Z;a_j)$. By calculating its expectation, we have:
\begin{equation}\label{ineq:expect-varimax-U}
    \begin{split}
        \bbE v(\hat Z; R) & = \bbE\sum_{j=1}^k v(Z;a_j) = 3k + (\kappa - 3)\sum_{i=1}^k\sum_{j=1}^k a_{ij}^4 \\
        & = \kappa k + (\kappa - 3)\sum_{i=1}^k\left(\sum_{j=1}^k a_{ij}^4 - 1\right) \overset{\text{Lemma \ref{lm:metric-equivalence}}}{\leq}\kappa k - \frac{(\kappa - 3)\delta^2}{16}k.
    \end{split}
\end{equation}
Meanwhile, note that
\begin{equation}\label{ineq:expect-varimax-V}
    \begin{split}
        \bbE v(\hat Z,R^*) & = \kappa k.
    \end{split} 
\end{equation}
Then we have
\begin{equation*}
    \begin{split}
        v(\hat Z; R^*) - v(\hat Z; R) &= \left(\bbE v(\hat Z; R^*) - \bbE v(\hat Z; R)\right) + \left(v(\hat Z; R^*) - \bbE(\hat Z; R^*)\right) - \left(v(\hat Z; R) - \bbE(\hat Z; R)\right) \\
        & \overset{\eqref{ineq:expect-varimax-U},\eqref{ineq:expect-varimax-V}}{\geq} \frac{(\kappa-3)\delta^2}{16}k - 2k\sup_{a:\|a\| \leq 1}\left|v(Z;a) - \bbE v(Z;a)\right|\\
        & \overset{\eqref{eq:def-delta},\eqref{ineq:concentration-varimax}}{>} 0.
    \end{split}
\end{equation*}
Therefore, $R$ is not a solution of varimax. In other words, given $\hat R = \argmax_{R \in \mathcal O(k)} v(\hat Z, R)$, one must have $\min_{P \in \mathcal P(k)}\left\|R^*R^\top - P\right\|_\tF < \delta k^{1/2}$, which completes the proof of sufficient condition.

Next, we prove the necessary condition. Let $Z_1 \in \bbR^{\lceil k/2 \rceil \times k}$ be the first $\lceil k/2 \rceil$ rows of the factor matrix $Z$. Consider the QR-decomposition of $Z_1^\top$: $Z_1^\top = U_1 R_1$, where $U_1$ is a $k$-by-$\lceil k/2 \rceil$ orthogonal matrix and $R_1$ is an $\lceil k/2 \rceil$-by-$\lceil k/2 \rceil$ upper triangular matrix. Let $U_{1\perp} \in \bbO_{k, k- \lceil k/2 \rceil}$ be the perpendicular subspace of $U_1$ (such that $U_1^\top U_{1\perp} = 0$), and denote $A = [U_1~U_{1\perp}]^\top R^*$. Then,
\begin{equation*}
    (\hat ZA^\top)_{[1:\lceil k/2 \rceil,:]} = Z_1 [U_1~U_{1\perp}] = [R_1^\top~O]
\end{equation*}
and it follows that
\begin{equation}\label{ineq:bad-solution}
    \begin{split}
        v(\hat Z, A) & = \frac{1}{n}\sum_{i=1}^n\sum_{j=1}^k \left((\hat Z A^\top)_{ij}\right)^4 \\
        & \geq \frac{1}{n}\left(\underbrace{\sum_{j=1}^{\lceil k/2 \rceil} (R_1)_{jj}^4}_{D_1} + \underbrace{\sum_{i=k+1}^n\sum_{j=1}^k \left((\hat Z A^\top)_{ij}\right)^4}_{D_2}\right).
    \end{split}
\end{equation}
Now we establish the probabilistic bounds for $D_1$ and $D_2$.
\begin{itemize}
    \item $D_1$. By random matrix theory~\cite{vershynin2010introduction}, we know that with probability at least $1-e^{-ck}$ that
\begin{equation}\label{ineq:lb-min-singular-value}
    \sigma_{\lceil k/2 \rceil}(Z_1) \geq \sqrt{k} - \sqrt{\lceil k/2 \rceil} - c\sqrt{k} \geq \sqrt{0.1k}.
\end{equation}
Note that $R_1$ shares the same non-zeros singular values with $X_1$. Thus it follows that
\begin{equation*}
    \prod_{j=1}^{\lceil k/2 \rceil} (R_1)^2_{jj} = |\det(R_1)|^2 = \prod_{j=1}^{\lceil k/2 \rceil} \sigma_j^2(Z_1) \overset{\eqref{ineq:lb-min-singular-value}}{\geq} (0.1k)^{\lceil k/2 \rceil}.
\end{equation*}
Then by Inequality of arithmetic and geometric means,
\begin{equation}\label{ineq:lb-D1}
    D_1 = \sum_{j=1}^{\lceil k/2 \rceil} (R_1)_{jj}^4 \geq \lceil k/2 \rceil \left(\prod_{j=1}^{\lceil k/2 \rceil} (R_1)^4_{jj}\right)^{1/\lceil k/2 \rceil} \geq 0.005 k^3.
\end{equation}
    \item $D_2$. Denote $B:= R^*A^\top$ and we can rewrite 
    \begin{equation*}
        \begin{split}
        D_2 & = \sum_{i=k+1}^n\sum_{j=1}^k \left((\hat Z A^\top)_{ij}\right)^4 = \sum_{i=k+1}^n \sum_{j=1}^k\left((ZB)_{ij}\right)^4 \\
        & = \sum_{j=1}^k \sum_{i=k+1}^n \left(\sum_{l=1}^k B_{lj}Z_{il}\right)^4.
        \end{split}
    \end{equation*}
    Then, conditioning on $A$ (or $B$), we have
    \begin{equation*}
        \begin{split}
            \bbE[D_2 | A]  &=  \bbE\left[\sum_{j=1}^k \sum_{i=k+1}^n \left(\sum_{l=1}^k B_{lj}Z_{il}\right)^4 \bigg| B\right] \\
            & = \sum_{j=1}^k (n-k)\bbE\left[\left(\sum_{k=1}^k B_{lj}Z_{1l}\right)^4 \bigg| B\right] \\
            & = (n-k)\left(3 k + (\kappa-3)\sum_{l=1}^k\sum_{j=1}^k B_{lj}^4 \right) \\
            & \geq 3(n-k)k.
        \end{split}
    \end{equation*}
    On the other hand, for each fixed $j \in [k]$, $\left\{\sum_{l=1}^k B_{lj}Z_{il}\right\}_{i={k+1}}^n$ are i.i.d. mean-0, variance-1 sub-Gaussian random variables (since $B$ is independent of the last $n-k$ rows of $Z$). By Lemma \ref{lm:concentration-fourth-moment}, we have
    \begin{equation*}
        \bbP\left(\sum_{i=k+1}^n \left(\sum_{l=1}^k B_{lj}Z_{il}\right)^4 - \bbE\left[\sum_{i=k+1}^n \left(\sum_{l=1}^k B_{lj}Z_{il}\right)^4\bigg |A \right]  < -C(\sqrt{nt}+t^2)\bigg| A\right) \leq 2e^{-ct}.
    \end{equation*}
    Taking $t = c_0k$ for some small constant $c_0$ and applying union bound for each $j \in [k]$, we know that
    \begin{equation*}
        \bbP\left(D_2 - \bbE[D_2|A] < -c_0\left(\sqrt{nk^3}+k^3\right)\bigg| A\right) \leq 2ke^{-ck}
    \end{equation*}
    and it follows that
    \begin{equation*}
        \bbP\left(D_2 < 3(n-k)k -c_0\left(\sqrt{nk^3}+k^3\right)\bigg| \A\right) \leq 2ke^{-ck}.
    \end{equation*}
    Integrate over the density of $A$ and we obtain
    \begin{equation}\label{ineq:lb-D2-bound}
        \bbP\left(D_2 < 3(n-k)k -c_0\left(\sqrt{nk^3}+k^3\right)\right) \leq 2ke^{-ck} \leq e^{-c'k}.
    \end{equation}
\end{itemize}
In conclusion, we proved that with probability at least $1-e^{-ck}$, 
\begin{equation}\label{ineq:lb-D1D2-bound}
    \begin{split}
        D_1 \geq 0.005k^3,\qquad D_2 \geq 3(n-k)k-c_0\left(\sqrt{nk^3} + k^3\right).
    \end{split}
\end{equation}
Since $n \leq ck^2$, by taking $c_0$ to be sufficiently small, we have
\begin{equation}\label{ineq:lb-construct-lb}
    v(\hat Z, A) \geq \frac{1}{n}\left(0.004k^3 + 3nk\right).
\end{equation}

Next we show that for any $R$ such that $\dist(R,R^*) \leq \delta$, $v(\hat Z, R)$ is upper bounded by the right-hand side of \eqref{ineq:lb-construct-lb} with high probability. Note that
\begin{equation}\label{ineq:v(Y,U)-bound-1}
    \begin{split}
        v(\hat Z; R) & = \bbE v(\hat Z; R^*) + (v(\hat Z, R^*) - \bbE v(\hat Z; R^*)) + \left(v(\hat Z, R) - v(\hat Z, R^*)\right)   \\
        & = \kappa k + \left(v(\hat Z; R^*) - \bbE v(\hat Z; R^*)\right) + \sup_{R: \dist(R,R^*) \leq \delta} \left|v(\hat Z, R) - v(\hat Z, R^*)\right|.
    \end{split}
\end{equation}

Following the same argument as \eqref{ineq:lb-D2-bound}, we know that with probability at least $1-e^{-ck}$ that
\begin{equation*}
    v(\hat Z; R^*) - \bbE v(\hat Z, R^*) \leq \frac{c_0}{n}\left(\sqrt{nk^3}+k^3\right).
\end{equation*}
In the meantime, 
\begin{equation}\label{ineq:lb:fluct-1}
    \begin{split}
        \sup_{R: \dist(R,R^*) \leq \delta} \left|v(\hat Z, R) - v(\hat Z, R^*)\right| & = \sup_{\substack{\tilde R \in \mathcal O(k): \\ \|\tilde R - I\|_\tF \leq \delta \sqrt{k}}}  \left|\left(v(Z;\tilde R) - v(Z;I)\right)\right| \\
        & \leq \sup_{\substack{\tilde R \in \mathcal O(k): \\ \|\tilde R - I\|_\tF \leq \delta \sqrt{k}}} \sum_{i=1}^k \left|\left(v(Z;\tilde R_{[i,:]}) - v(Z; e_i)\right)\right|,
    \end{split}
\end{equation}
where $e_i$ is the $i$-th canonical basis in $\bbR^k$. Let $\delta_i:= \|\tilde R_{[i,:]} - e_i\|_2$. Then, by the proof of \eqref{ineq:D1-bound} in Lemma \ref{lm:concentration-varimax}, for each $i \in [k]$, we have
\begin{equation}\label{ineq:lb:fluct-2}
    \left|v(Z;\tilde R_{i,:}) - v(Z; \tilde e_i) \right| \leq 6\delta_i \max_{j,l}Z_{jl}^4 \leq C\delta_i \log^2 n 
\end{equation}
hold with probability at least $1-n^{-4}$. Combining \eqref{ineq:lb:fluct-1} and \eqref{ineq:lb:fluct-2}, we obtain
\begin{equation}\label{ineq:lb:fluct-3}
    \begin{split}
        \sup_{R: \dist(R,R^*) \leq \delta} \left|v(\hat Z, R) - v(\hat Z, R^*)\right| & \leq C\log^2 n \sum_{i=1}^k \delta_i  \leq C\log^2 n \left(k\sum_{i=1}^k\delta_i^2\right)^{1/2} \\
        & = C\sqrt{k}\log^2 n \left\|\tilde R- I\right\|_\tF  \leq C\delta k\log^2 n.
    \end{split}
\end{equation}
Thus, when $n \lesssim \frac{k^2}{\log^2 n}$, we have
\begin{equation*}
    \sup_{R:\dist(R,R^*) \leq \delta} v(\hat Z,R) \overset{\eqref{ineq:v(Y,U)-bound-1},\eqref{ineq:lb:fluct-3}}{\leq} \kappa k + \frac{c_0\left(\sqrt{nk^3} + k^3\right)}{n} + \delta k\log^2 n < \frac{0.004k^3 + 3nk}{n} \overset{\eqref{ineq:lb-construct-lb}}{\leq} v(\hat Z, A).
\end{equation*}
Here, the second inequality is obtained by choosing $c_0$ to be sufficiently small and $\delta = 1/2$. This suggests that the solution of Varimax function $\hat R$ must satisfy:
\begin{equation*}
    \dist(\hat R,R^*) \geq \delta = 1/2.
\end{equation*}
\end{proof}

\begin{proof}[Proof of Lemma \ref{lm:metric-equivalence}]
Let $P^* = \argmin_{P \in \mathcal O(k)} \|A - P\|_\tF^2$. Without loss of generality, we can assume $P^* = I$ and $A_{11} \geq A_{22} \geq \ldots \geq A_{kk} \geq 0$. This suggests that 
\begin{equation}\label{ineq:metric-relation}
    t^2 < \sum_{i=1}^k \left((1-A_{ii})^2 + \sum_{j\neq i}A_{ij}^2\right) = 2\sum_{i=1}^k (1-A_{ii}) \leq 2\sum_{i=1}^k (1-A_{ii}^2).
\end{equation}
It suffices to show that for each $i \in [k]$,
\begin{equation*}
    \sum_{j=1}^k A_{ij}^4 - 1 \leq -\frac{1}{8}(1-A_{ii}^2).
\end{equation*}
Now consider the following two situations for each $i \in [k]$:
\begin{itemize}
    \item $A_{ii}^2 \geq 1/16$. 
    \begin{equation*}
        \begin{split}
        \sum_{j=1}^k A_{ij}^4 - 1 & \leq A_{ii}^4 + \left(\sum_{j\neq i} A_{ij}^2\right)^2 - 1 \\
        & = A_{ii}^4 + \left( 1-A_{ii}^2\right)^2 - 1 \\
        & = 2 A_{ii}^2\left(A_{ii}^2 - 1\right)  \leq - \frac{1}{8}(1-A_{jj}^2).
        \end{split}
    \end{equation*}
    \item $A_{ii}^2 < 1/16$. We claim that $\max_j A_{ij}^2 < 3/4$. Suppose there exists $l \in [k]/i$ such that $A_{il}^2 \geq 3/4$. Since $I = \argmin_{P \in \mathcal P(k)} \|A - P\|_\tF^2$, we must have
    \begin{equation}\label{ineq:lm1-best-D}
        \|A - I\|_\tF^2 \leq \|A - D^{(i, l)}\|_\tF^2,
    \end{equation}
    where $D^{(i,l)}$ is the permutation matrix such that $D^{(i,l)}_{il} = \text{sgn}(A_{il})$, $D^{(i,l)}_{li} = \text{sgn}(A_{lj})$ and $D^{(i,l)}_{jj} = 1$ for $j \in [k]/\{i,l\}$. This yields
    \begin{equation*}
        (A_{ll}-1)^2 + (A_{ii}-1)^2 + A_{il}^2 + A_{li}^2 \leq (|A_{li}|-1)^2 + (|A_{il}|-1)^2 + A_{ii}^2 + A_{ll}^2,
    \end{equation*}
    which is equivalent to 
    \begin{equation*}
        |A_{li}| + |A_{il}| \leq A_{ll} + A_{ii}.
    \end{equation*}
    Then it follows that
    \begin{equation*}
        A_{ll} \geq |A_{il}| - A_{ii} \geq \frac{2\sqrt{3}-1}{4}.
    \end{equation*}
    Consequently, we get
    \begin{equation*}
        1 \geq A_{ll}^2 + A_{il}^2 \geq \frac{13-4\sqrt{3}}{16}+\frac{3}{4} > 1.
    \end{equation*}
    This contradiction shows that we must have $\max_j A_{ij}^2 < 3/4$. Therefore,
    \begin{equation*}
        \sum_{j=1}^k A_{ij}^4 - 1 \leq \max_j A_{ij}^2 \sum_{j=1}^k A_{ij}^2 - 1 = -\frac{1}{4} \leq -\frac{1}{4}(1-A_{ii}^2).
    \end{equation*}
\end{itemize}
The proof is finished by combining the two situations. 
\end{proof}

\begin{proof}[Proof of Lemma \ref{lm:concentration-varimax}]
We first construct an $\varepsilon$-net $\cN = \left\{x^{(1)},\ldots,x^{(N)}\right\} \subset \mathbb S^{k-1}$ such that for any $a \in \mathbb S^{k-1}$, there exists some $i_a \in [N]$ with $\|a - x^{(i_a)}\| \leq \varepsilon$. By \cite{vershynin2010introduction}, we can choose such an $\varepsilon$-net with $N \leq (3/\varepsilon)^k$. Then, by triangle inequality,
\begin{equation*}
    \begin{split}
        & \sup_{a \in  \mathbb S^{k-1}} \left|v(Z;a) - \bbE v(Z;a)\right| \\
        \leq & \sup_{a \in  \mathbb S^{k-1}} \left(\left|\bbE v(Z;x^{(i_a)}) - \bbE v(Z;a)\right| + \left|v(Z;a) - v(Z;x^{(i_a)})\right| + \left|v(Z;x^{(i_a)}) - \bbE v(Z;x^{(i_a)})\right| \right) \\ 
        \leq & \underbrace{\sup_{a \in  \mathbb S^{k-1}}\left|\bbE v(Z;x^{(i_a)}) - \bbE v(Z;a)\right|}_{D_1} + \underbrace{\sup_{a \in  \mathbb S^{k-1}} \left|v(Z;a) - v(Z;x^{(i_a)})\right|}_{D_2} +   \underbrace{\max_{i \in N} \left|v(Z;x^{(i)}) - \bbE v(Z;x^{(i)})\right|}_{D_3}.
    \end{split}
\end{equation*}
We specify $\varepsilon=1/n$ and $N \leq (3n)^k$. Next we obtain deterministic bounds for $D_1$ and $D_2$ and a probabilistic bound for $D_3$.
\begin{itemize}
    \item $\sup_{a \in  \mathbb S^{k-1}}\left|\bbE v(Z;x^{(i_a)}) - \bbE v(Z;a)\right|$. Note that for any fixed $a \in \mathbb S^{k-1}$, one can calculate that
    \begin{equation*}
        \bbE v(Z;a) = \bbE\left(\sum_{j=1}^k a_jX_{1j}\right)^4 = (\kappa - 3)\sum_{j=1}^k a_j^4 + 3.
    \end{equation*}
    Therefore, 
    \begin{equation}\label{ineq:D2-bound}
        \begin{split}
            \left|\bbE v(Z;x^{(i_a)}) - \bbE v(Z;a)\right| & = (\kappa-3)\left|\sum_{j=1}^k \left(a_j^4 - (x^{(i_a)}_j)^4\right)\right| \\
            \leq & (\kappa-3) \left(\max_j a_j^2 + (x^{(i_a)}_j)^2\right) \cdot \sum_{j=1}^k\left| a_j^2 - (x^{(i_a)}_j)^2\right| \\
            \leq & 2(\kappa-3)\sum_{j=1}^k\left| a_j^2 - (x^{(i_a)}_j)^2\right|. 
        \end{split}
    \end{equation}
    
    Denote $\delta = a - b$ for any $b \in \mathbb S^{k-1}$. Note that
    \begin{equation*}
        \sum_{j=1}^k\left|a_j-b_j\right| = \sum_{j=1}^k\left|b_j^2+\delta_j^2 + 2b_j\delta_j-b_j^2\right| \leq \|\delta\|_2^2 + 2\sum_{j=1}^k \left|b_j\delta_j\right| \leq \|\delta\|_2^2 + 2\|\delta\|_2,
    \end{equation*}
    where we use Cauchy-Schwardz inequality and the fact $\|b\|=1$ to obtain the last inequality. Now we replace $b$ with $x^{(i_a)}$, then $D_1$ can be further bounded as
    \begin{equation}\label{ineq:D1-bound}
    \begin{split}
        D_1 & \leq 2(\kappa-3)\sum_{j=1}^k\left| a_j^2 - (x^{(i_a)}_j)^2\right| \leq 2(\kappa-3) \max_j |a_j+x_j^{i_a}| \cdot \sum_{j=1}^j \left|a_j - x_j^{(i_a)}\right| \\
        & \leq 12\varepsilon (\kappa - 3) \leq \frac{C}{n}.
    \end{split}
    \end{equation}

    \item $\sup_{a \in  \mathbb S^{k-1}} \left|v(Z;a) - v(Z;x^{(i_a)})\right|$. For any $a,b \in \mathbb S^{k-1}$,  
    \begin{equation*}
        \begin{split}
            \left|v(Z;a) - v(Z;b)\right| & = \frac{1}{n}\left|\sum_{i=1}^n\left(\left(\sum_{j=1}^k a_jZ_{ij}\right)^4 - \left(\sum_{j=1}^kb_jZ_{ij}\right)^4\right) \right| \\
            & \leq \frac{1}{n}\max_{i}\left|\left(\sum_{j=1}^k a_jZ_{ij}\right)^2 + \left(\sum_{j=1}^kb_jZ_{ij}\right)^2\right| \cdot \sum_{i=1}^n \left|\left(\sum_{j=1}^k a_jZ_{ij}\right)^2 - \left(\sum_{j=1}^kb_jZ_{ij}\right)^2 \right| \\
            & \leq \frac{2}{n}\left(\max_{i} \sum_{j=1}^k Z_{ij}^2\right) \cdot \max_i \left|\sum_{j=1}^k (a_j+b_j) Z_{ij}\right| \cdot \sum_{i=1}^n\sum_{j=1}^k|(a_j-b_j)Z_{ij}|\\
            & \leq 4k^{3/2} (\max_{i,j}|Z_{ij}|)^4 \cdot \sum_{j=1}^k |a_j-b_j|  \leq 4k^{2} (\max_{i,k}|Z_{ik}|)^4 \cdot \sqrt{\sum_{j=1}^k (a_j-b_j)^2} \\
            & \leq 4k^2 (\max_{i,j}|X_{ij}|)^4 \varepsilon\leq \frac{4k^2 (\max_{i,j}|X_{ij}|)^4}{n}.
        \end{split}
    \end{equation*}
    
    \item $\max_{i \in N} \left|v(Z;x^{(i)}) - \bbE v(Z;x^{(i)})\right|$. By Lemma \ref{lm:concentration-fourth-moment}, for fixed $x^{(l)}$, since $Y_i^{(l)} := \sum_{j=1}^k x^{(l)}_jX_{ij}$ are i.i.d. mean-zero, variance one, sub-Gaussian random variables, we have
\begin{equation*}
    \bbP\left(v(Z;x^{(j)}) - \bbE v(Z;x^{(j)}) \geq C\left(\sqrt{\frac{t}{n}} + \frac{t^2}{n}\right)\right) \leq e^{-ct}.
\end{equation*}
Taking $t = Ck\log n$ and applying union bounds on $x^{(l)}$ for each $l \in [N]$, it follows that
\begin{equation}\label{ineq:D3-bound}
    \bbP\left(D_3 \leq C\left(\sqrt{\frac{k\log n}{n}} + \frac{k^2\log^2 n}{n}\right)\right) \geq 1-e^{-ck\log n}.
\end{equation}
In the meantime, by the tail bound of sub-Gaussian Extreme values, we have
\begin{equation}\label{ineq:X-max-bound}
    \bbP\left(\max_{i,k} |X_{i,k}| \leq \sqrt{10\log(nk)} \right) \leq 1-\sqrt{\frac{2}{\pi \log n}}n^{-4}.
\end{equation}
\end{itemize}
Combining \eqref{ineq:D1-bound}, \eqref{ineq:D2-bound}, \eqref{ineq:D3-bound} and \eqref{ineq:X-max-bound}, we know that with probability at least $1-n^{-ck}-n^{-4}$ that
\begin{equation*}
    \sup_{a \in  \mathbb S^{k-1}} \left|v(Z;a) - \bbE v(Z;a)\right| \lesssim \sqrt{\frac{k\log n}{n}} + \frac{k^2\log^2 n}{n}.
\end{equation*}
Now the proof is finished.
\end{proof}

\begin{proof}[Proof of Lemma \ref{lm:concentration-fourth-moment}]
Note that $\{X_i^4\}_{i=1}^n$ are sub-Weibull($1/2$) random variables and the result simply comes from the concentration developed by \cite[Theorem 3.1]{kuchibhotla2018moving}.
\end{proof}

\bibliographystyle{plain}
\bibliography{reference}

\end{document}